\documentclass[]{birkjour}
\usepackage{amsmath,amsthm,amsopn,amssymb,mathrsfs,times,newtxtext,newtxmath,upref, systeme, amsmath,xcolor}
\usepackage[mathscr]{eucal}

\usepackage{enumitem}
\setlist{label={$$\roman{enumi}\kern1pt$)$}}

\newtheorem{thm}{Theorem}[section]
\newtheorem{prop}[thm]{Proposition}
\newtheorem{cor}[thm]{Corollary}
\newtheorem*{cor*}{Corollary}
\newtheorem{lema}[thm]{Lemma}
\newtheorem*{lema*}{Lemma}

\numberwithin{equation}{section}
\theoremstyle{definition}
\newtheorem*{Def}{Definition}

\newtheorem{Example}{Example}
\newtheorem{obs}{Remark}

\newcommand{\PI}[2]{\left\langle \,#1 , #2\, \right\rangle}

\newcommand{\St}{\mathcal{S}}

\newcommand{\HH}{\mathcal{H}}
\newcommand{\KK}{\mathcal{K}}
\newcommand{\M}{\mathcal{M}}

\newcommand{\N}{\mathcal{N}}

\newcommand{\G}{\Gamma}
\newcommand{\mc}[1]{\mathcal{#1}}

\newcommand{\ol}{\overline}
\newcommand{\clran}{\ol{\mathrm{ran}}\,}

\newcommand{\cldom}{\ol{\mathrm{dom}}\,}
\newcommand{\clmul}{\ol{\mathrm{mul}}\,}
\newcommand{\reg}{\mathrm{reg}}
\newcommand{\ra}{\rightarrow}
\DeclareMathOperator{\Sp}{Mp}
\DeclareMathOperator{\SP}{MP}
\DeclareMathOperator{\Mp}{Mp}
\DeclareMathOperator{\MP}{MP}

\newcommand{\PMN}{P_{\M,\N}}

\DeclareMathOperator{\re}{reg}
\DeclareMathOperator{\sing}{sing}

\DeclareMathOperator{\ran}{ran}
\DeclareMathOperator{\dom}{dom}

\DeclareMathOperator{\mul}{mul}

\newcommand{\pl}{{\mathbin{/\mkern-3mu/}}}

\newcommand{\lr}{\mathrm{lr}}

\begin{document}
\title{Semiclosed multivalued projections}

\author[Arias]{M.~Laura~Arias}

\address{
	Instituto Argentino de Matem\'atica ``Alberto P. Calder\'on'' \\
	CONICET\\
	Saavedra 15, Piso 3\\
	(1083) Buenos Aires, Argentina \\[5pt]
	Facultad de Ingenier\'{\i}a, Universidad de Buenos Aires\\
	Paseo Col\'on 850 \\
	(1063) Buenos Aires, Argentina}
\email{lauraarias@conicet.gov.ar}

\author[Contino]{Maximiliano Contino}

\address{Depto. de An\'alisis Matem\'atico, Facultad de Matem\'aticas, Universidad Complutense de Madrid \\ Plaza de Ciencias 3 \\ (28040) Madrid, Spain \\ [5pt]
	Instituto Argentino de Matem\'atica ``Alberto P. Calder\'on'' \\
	CONICET\\
	Saavedra 15, Piso 3\\
	(1083) Buenos Aires, Argentina \\[5pt]
	Facultad de Ingenier\'{\i}a, Universidad de Buenos Aires\\
	Paseo Col\'on 850 \\
	(1063) Buenos Aires, Argentina}
\email{mcontino@fi.uba.ar}

\author[Maestripieri]{Alejandra~Maestripieri}

\address{
	Instituto Argentino de Matem\'atica ``Alberto P. Calder\'on'' \\
	CONICET\\
	Saavedra 15, Piso 3\\
	(1083) Buenos Aires, Argentina \\[5pt]
	Facultad de Ingenier\'{\i}a, Universidad de Buenos Aires\\
	Paseo Col\'on 850 \\
	(1063) Buenos Aires, Argentina}
\email{amaestri@fi.uba.ar}

\author[Marcantognini]{Stefania Marcantognini}

\address{
	Instituto Argentino de Matem\'atica ``Alberto P. Calder\'on'' \\
	CONICET\\
	Saavedra 15, Piso 3\\
	(1083) Buenos Aires, Argentina \\[5pt]
	Universidad Nacional de General Sarmiento -- Instituto de Ciencias
	\\ Juan Mar\'ia Gutierrez 1150 
	\\ (1613) Los Polvorines, Pcia. de
	Buenos Aires, Argentina}
\email{smarcantognini@ungs.edu.ar}


%

\begin{abstract} 
A multivalued projection is an idempotent linear relation with invariant domain. We characterize multivalued projections that are operator ranges (called semiclosed) and provide several formulae of them. 
Moreover, we study the decomposability and continuity of multivalued projections, and describe nilpotent relations. 

\end{abstract}

\maketitle
	
	\section{Introduction}
\label{sec:intro}

Linear relations are a natural generalization of linear operators. As with operators, the notions of idempotents and nilpotents, but now multivalued, play a central role in their study. A linear relation acting between Hilbert spaces is any subspace of their product, much as an operator can be characterized by its graph. 

A \emph{multivalued projection}, or  \emph{semi-projection}, $E$ acting on a Hilbert space $\HH$ is an idempotent linear relation of $\HH \times \HH$ (i.e., $E^2=E$) with invariant domain. A multivalued projection is completely determined by its range and kernel.  We refer to \cite{Cross1}, \cite{Cross} and \cite{Labrousse} for a detailed account on the subject.

Linear relations that are operator ranges (commonly referred to as semiclosed linear relations) can be described as \emph{quotients} of bounded operators  in the sense of linear relations \cite{Hassi}, and this provides extra structure in describing the properties of such relations. 

An important example of a  multivalued projection  that is an operator range is one having range that of a given operator and kernel its \emph{de Brange-Rovnyak complement} \cite{Bolotnikov}. The multivalued part is then their intersection, called the ``overlapping space''; see Example \ref{Example3}. Another example of a semiclosed multivalued projection  appears when solving operator least squares problems with a selfadjoint or positive (semi-definite) weight in a Hilbert space \cite{Matrix}. This is explained in Example \ref{Example4}.

It is easy to check that a multivalued projection is semiclosed if and only if its range and kernel are both operator ranges, where the special case for operators can be found in  \cite{Ota} and  \cite{Semiclosed}.

Any semiclosed multivalued projection $E$ can be written as a direct componentwise sum of an operator (identified with its graph)  with domain equal to the domain of the multivalued projection and $E_{\mul}:=\{0\} \times \mul E,$ where $\mul E$ is the multivalued part of $E.$ Moreover, any such multivalued projection is quasi-affine to a multivalued projection with domain $\cldom E$ and operator  part a positive contraction.
This generalizes a result by Ando for densely defined closed projections, which states that any densely defined closed projection $E$ acting on a Hilbert space is quasi-affine to an orthogonal projection $P$ \cite[Theorem 2.3]{Ando}, meaning that there is positive injective bounded operator, intertwining $E$ and $P$. Moreover, when $\mul E$ is closed in $\dom E,$ then $E$ is quasi-affine to a closed multivalued projection with operator part an orthogonal projection. 

The paper is organized as follows: notation and background material are given in Section \ref{sec:preliminaries}. In Section \ref{sec:three} we start by gathering some known results about multivalued projections. Then we characterize multivalued projections and nilpotent relations, two examples of linear relations with ranges contained in their domains. We end the section by generalizing two formulae for projections in terms of orthogonal projections onto their range and kernel: the formula given by Greville \cite{Greville} for projectors in finite dimensional spaces and the one given by Pt\'ak \cite{Ptak} for projections in Hilbert spaces. In Section \ref{sec:four} we focus our attention on decomposability and continuity of multivalued projections and study some distinguished orthogonal decompositions for multivalued projections, characterizing those whose operator parts in their Lebesgue decomposition are projections. Section \ref{sec:five} is devoted to multivalued projections that are operator ranges. We finish the section with a formula that generalizes the one given by Ando for densely defined closed projections in Hilbert spaces \cite{Ando}.

\section{Preliminaries}
	\label{sec:preliminaries}
	
In this paper $\HH,$ $\KK$ and  $\mc E$  are complex and separable Hilbert spaces. The space of bounded linear operators from $\HH$ to $\KK$ is denoted by $L(\HH, \KK)$ and $L(\HH)$ when $\HH=\KK.$ 
Given a closed subspace $\M$ of $\HH,$  $P_\M $ is the orthogonal projection onto $\M.$  The set of orthogonal projections is denoted by $\mc P.$
The direct sum of two subspaces $\M$ and $\N$ of $\HH$ is indicated by $\M \dotplus \N,$ and $\M \oplus \N$ if $\M \subseteq \N^{\perp}$. In addition, if $\N \subseteq \M,$ $\M\ominus\N$ means $\M\cap\N^\bot.$
An \emph{operator range} is a linear subspace of $\HH$ that is the range of some bounded operator on $\HH$ \cite{Filmore}. The following properties of operator ranges are proved in \cite[Proposition 2.3.3 and Corollary 2.3.1]{Labrousse1980}.

\begin{prop} \label{ORLabrousse} Let $\M$ and $\N$ be operator ranges of $\HH$ such that $\M+\N$ is closed. Then the following hold:
	\begin{enumerate}
		\item [1.] $\ol{\M\cap \N}=\ol{\M}\cap \ol{\N}.$
		\item[2.] $(\M\cap \N)^{\perp}=\M^{\perp}+\N^{\perp}.$
	\end{enumerate}
\end{prop}
A linear relation from  $\HH$ into $\KK$ is a  subspace  of  $\HH \times \KK.$ The set of linear relations from $\HH$ into $\KK$ is denoted by $\lr(\HH,\KK),$ and $\lr(\HH)$ when $\HH=\KK.$ Given $T\in \lr(\HH,\KK),$ $\dom T,$  $\ran T$ and $\ker T$ denote the domain,  range and kernel of $T,$ respectively. The multivalued part of $T$ is defined by $\mul T:=\{y\in \KK: (0,y)\in T\}.$ If $\mul T=\{0\},$ $T$ is (the graph of) an operator. 
The inverse of $T$ is the relation $T^{-1}:=\{(y,x) : (x,y) \in T\}$. Thus,  $\dom T^{-1}=\ran T$ and $\mul T^{-1} =\ker T.$ 
	
For $T, S \in \lr(\HH,\KK),$  $T \ \hat{+}  \ S$ stands for the sum of $T$ and $S$ as subspaces.
The notations  $T\hat{\dotplus}S$ and $T \hat{\oplus} S$ are self-explanatory. It is useful to note that $(T \ \hat{+} \ S)^{-1}=T^{-1} \ \hat{+} \ S^{-1}.$

The sum of $T+S$ is the linear relation defined by
$$T+S:=\{(x,y+z): (x,y ) \in T \mbox{ and } (x,z) \in S\}.$$
If $R \in \lr(\KK,\mc E)$  the product $RT$ is the linear relation from $\HH$ to $\mc E$ defined by
$$RT:=\{(x,y): (x,z) \in T \mbox{ and } (z,y) \in R \mbox{ for some } z \in \mc{K}\}.$$ It  holds that	$(RT)^{-1}=T^{-1}R^{-1}$.

The following results will be used throughout the paper without mention.
 
\begin{lema}[{\cite[2.02]{Arens}, \cite[Proposition 1.21]{Labrousse}}] \label{lemalr} Let $S, T \in \lr(\HH,\KK).$ Then $S=T$ if and only if $S\subseteq T,$ $\dom T\subseteq\dom S$ and $\mul T \subseteq \mul S.$
\end{lema}

\begin{lema} \label{sumasubs} Let $T, S \in \lr(\HH,\KK),$ $R \in \lr(\KK,\mc E)$ and $F\in \lr(\mc E, \HH)$. Then
	\begin{enumerate}
		\item[1.] $RT \ \hat{+} \ RS \subseteq R(T \ \hat{+} \ S) $ and the equality holds if $\ran T\subseteq \dom R$ or $\ran S\subseteq \dom R$.
		\item[2.] $TF \ \hat{+} \ SF \subseteq (T \ \hat{+} \ S)F $ and the equality holds if  $\dom T \subseteq \ran F$ or $\dom S\subseteq \ran F$.
	\end{enumerate}
	\end{lema}
\begin{proof} $1.$ The proof of the inclusion is trivial. Assume that $\ran T\subseteq \dom R$ and let $(x,y)\in R(T \ \hat{+} \ S).$ Thus, there exist $z=z_1+z_2\in\KK$ and $x_1,x_2\in \HH$ such that $x=x_1+x_2,$ $(x_1,z_1)\in T$, $(x_2,z_2)\in S$ and $(z,y)\in R.$ Since $\ran T\subseteq \dom R$,  there exists $w\in \mc E$ such that $(z_1,w)\in R$ then $(x_1,w)\in RT.$  Also, $(z_2,y-w)\in R$ and so  $(x_2,y-w)\in RS$. Therefore, $(x,y)=(x_1,w)+(x_2,y-w)\in RT \ \hat{+} \ RS$ and the equality holds.
	 
$2.$ Take inverses and use $1.$
\end{proof}

Given a subspace $\M$ of $\HH$, $I_\M:=\{(u,u): u\in\M\}$ and $0_\M:=\M \times \{0\}.$ When $\M=\HH$ we write $I$ and $0$ instead. The following identities can be easily proved
\begin{align} \label{TT1}
TT^{-1}T&=T \mbox{ and } T^{-1}T=I_{\dom T} \ \hat{+} (\{0\} \times  \ker T)=I_{\dom T} \ \hat{+} (\ker T \times  \{0\}),
\end{align}
where we use that $\ker T \subseteq \dom T$ for the last equality.

Set
$$
T(\mc{M}):=\{y : (x,y) \in T \mbox{ for some } x \in \mc{M} \},
$$  
and for any $x \in \dom T,$ $Tx:=T(\{x\}).$

The \emph{closure} $\ol{T}$ of $T$ is the closure of $T$ in $\HH \times \KK$ endowed with the product topology. Thus, the relation $T$ is  \emph{closed} when $T=\ol{T}.$

The \emph{adjoint} of $T\in \lr(\HH,\KK)$ is the linear relation from $\KK$ to $\HH$ defined by
$$T^*:=\{(x,y) \in \KK \times \HH: \PI{g}{x}=\PI{f}{y} \mbox{ for all } (f,g) \in T \}.$$ 
	The adjoint of $T$
	is  a closed linear relation,  $\ol{T}^*=T^*$ and $T^{**}:=(T^*)^*=\ol{T}.$ It holds that $\mul T^* =(\dom T)^{\perp}$ and $\ker T^* =(\ran T)^{\perp}.$ Therefore, if $T$ is closed both $\ker T$ and $\mul T$ are closed subspaces.

\subsection{Decompositions of linear relations}
	Next we outline some basics about decompositions of linear relations from the work on the subject by Hassi et al., see \cite{Hassi,Hassi2,hassi2018lebesgue}.
	
	Given $T \in \lr(\HH, \KK),$ a subspace $\tilde{T}$ of $T$ is called an \emph{ operator part} of $T$ if $\tilde{T}$ is an operator with $\dom \tilde{T}=\dom T$ and $\tilde{T}\subseteq T.$ In this case, $T=\tilde{T} \ \hat{+} \ T_{\mul},$ where  $T_{\mul}:=\{0\} \times \mul T$. The relation $$T_0:=T \cap (\cldom T \times \cldom T^*)$$  is a (closable) operator from $\cldom T$ to $\cldom T^*$ contained in $T$ but,  in general, $\dom T_0\subsetneq \dom T$. We say that $T$ is \emph{decomposable} if $\dom T_0 = \dom T$, i.e., if $T_0$ is an operator part of $T$ \cite{Hassi2}. In other words, $T$ is decomposable if  $T$  admits the componentwise sum decomposition
\begin{equation} \label{Tdecom}
	T=T_0 \ \hat{\oplus} \ T_{\mul}.
\end{equation} 

On the other hand, a linear relation $T$ is said to have a \emph{distinguished orthogonal range decomposition} if $T = T_1 + T_2$ with $T_1$ an operator, $\dom T_1 = \dom T_2 = \dom T$ and $\ran T_1 \bot \ran T_2$. By \cite[Corollary 3.6]{hassi2018lebesgue}, $T = T_1 + T_2$ is a distinguished orthogonal range decomposition of $T$ if and only if there exists $Q \in \mc P$ such that $\mul T\subseteq \ker Q,$ and in this case, $T_1=QT$ and $T_2=(I-Q)T.$  

Among the distinguished orthogonal range decompositions, there are  two with some extremal properties \cite{hassi2018lebesgue}. Consider $P:=P_{\cldom T^*}$  and define the \emph{regular part} $T_{\re}$ and the \emph{singular part} $T_{\sing}$ of $T$ by
$$T_{\re}:=PT \ \ \mbox{ and } \ \ T_{\sing}:=(I-P)T.$$

The terminology refers to the notions of regular and singular linear relations. $T \in L(\HH, \KK)$ is said to be \emph{regular} (or closable) if $\ol{T}$ is an operator; \emph{singular} if $\ol{T}$ is the (Cartesian) product of closed subspaces in $\HH$ and $\KK.$
Clearly, $T$ has the distinguished orthogonal range decomposition
\begin{equation}\label{Treg+Tsing}
T = T_{\reg} + T_{\sing}.
\end{equation}
This is the \emph{Lebesgue decomposition} of $T.$
The regular and singular parts verify that $\ol{T_{\reg}}$ is an operator and $\ol{T_{\sing}}=\dom \ol{T} \times \mul \ol T$ \cite[Theorem 4.1]{hassi2018lebesgue}. If $Q:=P_{\mul T^{\perp}}$ and $T_m:= QT,$ then
\begin{equation}\label{Tm}
T = T_m+ (I-Q)T
\end{equation}
is also a  distinguished orthogonal range decomposition of $T$, known as the \emph{weak Lebesgue decomposition} of $T.$

Despite the different nature of the decompositions (\ref{Tdecom}) and (\ref{Treg+Tsing}) (or (\ref{Tm})),  the operator terms may be the same as the following theorem shows. 

\begin{thm}[{\cite[Theorems 3.10 and 3.18]{Hassi2}}] \label{decomthm} Let $T \in \lr(\HH, \KK).$ The following are equivalent:
	\begin{enumerate}
		\item $T$ is decomposable;
		\item $\ran T_{\sing} \subseteq \mul T;$
		\item $T_0=T_{\re};$
		\item $T_0= T_m$.
	\end{enumerate}
If any of these conditions hold then $\mul \ol{T}=\clmul T.$
\end{thm}

\begin{Def} $T \in \lr(\HH,\mc K)$ is \emph{continuous} if for any neighbourhood $V,$ $T^{-1}(V)$ is a neighbourhood in $\dom T.$ 
\end{Def}

\begin{prop}[{\cite[Prop. 3.1]{Cross}}] A relation $T \in \lr(\HH,\KK)$ is \emph{continuous} if and only if $T_m$ is bounded. 
\end{prop}

By Theorem \ref{decomthm}, a decomposable linear relation $T$ is continuous if and only if $T_0$
is bounded.

\begin{prop} [{\cite[Theorem 3.2]{Cross}}] \label{Tcont}  Let $T\in\lr(\HH, \KK)$ be closed. Then $T$ is continuous if and only if $\dom T$ is closed.
\end{prop}

\begin{prop} [{\cite[Proposition 3.5]{Hassi2}}] \label{bounded3} 
 Let $T \in \lr(\HH, \KK).$ Then $T_{\re}$ is a bounded operator if and only if $\dom T^*$ is closed.
\end{prop}

\section{Multivalued projections and nilpotents} 	\label{sec:three}

Following \cite{Ota}, we now present two  types of linear relations with domains containing their ranges: the multivalued projections and the multivalued nilpotents. 

\begin{Def} Let $E \in \lr(\HH)$ such that $\ran E \subseteq \dom E.$ We say that $E$ is a \emph{multivalued projection} if $E$ is \emph{idempotent}, that is $E^2=E;$ and  a \emph{multivalued nilpotent} if
\begin{equation} \label{eqnil}
E^2=\dom E \times \mul E.
\end{equation}
\end{Def}

Notice that $E$ is a multivalued nilpotent relation if and only if $$E^2=0_{\dom E} \ \hat{+} \ (\{0\} \times \mul E).$$
If $E$ is a multivalued projection (respectively a multivalued nilpotent) with $\mul E=\{0\},$ then $E$ is a \emph{projection} (respectively a \emph{nilpotent}).

There are idempotent linear relations which are not multivalued projections. For example, if $E$ is a projection then, by \cite[Corollary 4.15]{idem}, $E^{-1}$ is an idempotent relation such that $\ran E^{-1} =\dom E=\ran E \dotplus \ker E$ and  $\dom E^{-1}=\ran E.$
Therefore if $\ker E \not = \{0\},$ $E^{-1}$ is not a multivalued projection.

Also, there are relations satisfying \eqref{eqnil} whose domains do not contain their ranges. For instance, let $\M \not =\{0\}$ be a subspace and set $E=\{0\} \times \M$ so that  $E$ satisfies \eqref{eqnil} but $\ran E=\M \not \subseteq \{0\}=\dom E.$

Multivalued projections were introduced by Cross and Wilcox in \cite{Cross} and later studied by Labrousse in \cite{Labrousse}. Multivalued projections preserve many properties of projections, for instance, they are fully described by their ranges and kernels. In the sequel, $\Sp(\HH)$ denotes the set of multivalued projections on $\HH$ and $\SP(\HH)$ stands for the subset of $\Sp(\HH)$ of closed multivalued projections. 

\begin{prop}[\cite{Cross,Labrousse}] \label{prop2.2} $E \in \Sp(\HH)$ if and only if 
$E=I_{\ran E} \ \hat{+} \ (\ker E \times \{0\}).$
\end{prop}
From now on $\M, \N$ are subspaces of $\HH.$  In view of the above proposition write  $$P_{\M, \N}:=I_\M \ \hat{+} \ (\N \times \{0\}).$$ 

Thus, $\PMN$ denotes the multivalued projection with range $\M$ and kernel $\N.$ It is easy to check that $\dom \PMN=\M+\N$ and $\mul \PMN=\M\cap \N.$ When $\PMN$ is a projection, we write $P_{\M \pl \N},$ and  $P_\M$ if $\N=\M^{\perp}.$

	\begin{Example} \label{remark1} Given $T \in \lr(\HH)$ we can express \eqref{TT1}  in terms of multivalued projections since 
	$$T^{-1}T=P_{\dom T, \ker T} \mbox{ and }  TT^{-1}=P_{\ran T, \mul T}.$$ Moreover $T^{-1}$ is the unique solution of the system of equations 
	\begin{equation} \label{system}
		XT=P_{\dom T, \ker T}, TX=P_{\ran T, \mul T},  XTX=X.
	\end{equation}
In fact, by \eqref{TT1}, $T^{-1}$ is a solution. Suppose that $X \in \lr(\HH)$ is also a solution. Then $XT=T^{-1}T=(XT)^{-1}$ and $TX=TT^{-1}.$ Hence, $TXT=TT^{-1}T=T.$ Taking inverses, 
	$$T^{-1}=(TXT)^{-1}=(XT)^{-1}T^{-1}=XTT^{-1}=XTX=X.$$
\end{Example}

\medskip

\begin{prop} \label{lemamul} Let $T \in \lr(\HH).$ Then
\begin{enumerate}
	\item [1.] $T$ is a multivalued projection if and only if $I_{\ran T} \subseteq T.$
	\item [2.] $T$ is a multivalued nilpotent if and only if $\ran T \subseteq \ker T.$ 
\end{enumerate}	
\end{prop} 

\begin{proof} $1.$ If $T$ is a multivalued projection, by Proposition \ref{prop2.2}, $I_{\ran T} \subseteq T.$ Conversely, if $I_{\ran T} \subseteq T$ then $\ran T \subseteq \dom T$ and  $I_{\ran T} T \subseteq T^2,$ or $T\subseteq T^2.$ On the other hand, taking inverses, $I_{\ran T} \subseteq T^{-1}.$ Then $T^2=TI_{\ran T} T\subseteq T T^{-1} T=T $

$2.$ Suppose that $T^2=\dom T \times \mul T$ and $\ran T \subseteq \dom T.$ Given $x \in \ran T$, there exist $y,z\in \HH$ such that $(y,x), (x,z) \in  T$ or $(y,z) \in T^2.$ In this case, $z \in \mul T$ or $(0,z) \in T$. Then $(x,0) \in T$ or $x \in \ker T.$ Conversely,  if $\ran T \subseteq \ker T$ then $\ran T \subseteq \dom T,$ and  $T=P_{\ker T}T.$ On the other hand, it always holds that $TP_{\ker T}=(\ker T \oplus \ker T^{\perp}) \times  \mul T.$ Then 
$$T^2=TP_{\ker T}T=((\ker T \oplus \ker T^{\perp}) \times  \mul T)T=\dom T \times \mul T,$$ because $\ran T \subseteq \ker T.$ 
\end{proof}

\begin{prop} If $T \in \lr(\HH)$ is a multivalued projection (a multivalued nilpotent) then $T^*$ and $\ol{T}$ are multivalued projections (multivalued nilpotents, respectively).  
\end{prop}
\begin{proof} The result was proved for multivalued projections in \cite{Cross, Labrousse}; see \eqref{L2} below.
	
Let $T$ be a multivalued nilpotent. Since $\ol{T}=(T^*)^*$ we only need to show the result for $T^*.$
By Proposition \ref{lemamul}, $\ran T \subseteq \ker T$ and then $(\ker T)^{\perp} \subseteq (\ran T)^{\perp}=\ker T^*.$ On the other hand, $\ker T \subseteq \ker \ol{T}$ implies that $\clran T^*=(\ker \ol{T})^{\perp} \subseteq  (\ker T)^{\perp}.$ Therefore $\ran T^*\subseteq \ker T^*,$ and by  Proposition \ref{lemamul}  once again, $T^*$ is a multivalued nilpotent.
\end{proof}

For multivalued projections, the formulae
\begin{equation}\label {L2} 
P_{\M, \N}^*=P_{\N^{\perp},\M^{\perp}} \mbox{ and } \ol{P_{\M, \N}}=P_{\ol{\M}, \ol{\N}},
\end{equation} 
hold \cite{Cross, Labrousse}. Then $P_{\M, \N} \in \MP(\HH)$ if and only if $\M$ and $\N$ are closed.

\bigskip
In  \cite[Theorem 2]{Greville}, Greville proved that  if $\M$ and $\N$ are  (finite dimensional) complementary subspaces then 
\begin{equation}\label{Greville}
P_{\M \pl \N}=(P_{\N^{\perp}}P_\M)^{\dagger}. 
\end{equation} 
The same formula holds for closed subspaces in a Hilbert space $\HH$ such that $\ol{\M \dotplus \N}=\HH$ \cite{CorachMaestri}.

On the other hand, if $T$ is a closed operator, its Moore-Penrose inverse can be given in terms of (the linear relation) $T^{-1}$  \cite{Nashed,TA} as 
\begin{equation} \label{alvarezMP}
	T^{\dagger}:=P_{\ker T^{\perp}} T^{-1} P_{\ker {T^*}^{\perp}}.
\end{equation}

If $\M$ and $\N$ are closed subspaces such that $\ol{\M \dotplus \N}=\HH,$ the formula \eqref{alvarezMP} applied to $T=P_{\N^{\perp}}P_{\M}$ gives
$$(P_{\N^{\perp}}P_\M)^{\dagger}=P_{\M}(P_{\N^{\perp}}P_\M)^{-1}P_{\N^{\perp}}$$ because $(\ker( P_{\N^{\perp}}P_\M))^{\perp}=\M$ and $(\ker(P_\M P_{\N^{\perp}}))^{\perp}=\N^{\perp}.$ Then, if $P_{\M \pl \N}$ is a densely defined closed projection, by \eqref{Greville},
$$P_{\M \pl \N}=P_{\M}(P_{\N^{\perp}}P_\M)^{-1}P_{\N^{\perp}}.$$
A similar result holds for multivalued projections if an extra hypothesis is required. 

\begin{prop} \label{Grevilleprop} Let $\M, \N$ be subspaces of $\HH$ such that $\M\subseteq \N\oplus \N^{\perp}.$ Then
	$$\PMN=P_\M((I-P_\N)P_\M)^{-1}(I-P_\N).$$
\end{prop}
\begin{proof}  By \eqref{TT1} and the fact that $I_\M I_{\N \oplus \N^{\perp}}=I_\M$ we get that
$P_\M((I-P_\N)P_\M)^{-1}(I-P_\N)=$$P_\M{P_\M}^{-1}(I-P_\N)^{-1}(I-P_\N)$$=I_\M(I_{\N \oplus \N^{\perp}} \ \hat{+} \ (\{0\} \times  \N))=I_\M(I_{\N \oplus \N^{\perp}} \ \hat{+} \ (\N \times \{0\}))=I_\M \ \hat{+} \  (\N \times \{0\})=\PMN.$
\end{proof}

\medskip
As a corollary, we obtain another formula for closed multivalued projections that generalizes the one given by Pt\'ak \cite{Ptak} for bounded projections. If $\M$ and $\N$ are closed subspaces such that $\M \dotplus \N=\HH$ then $ (I-P_\N P_\M)$ is invertible and 
$$P_{\M \pl \N}=(I-P_\N P_\M)^{-1}P_{\N^{\perp}},$$ see also \cite[Theorem 3.3]{Greville}.

\begin{cor}[{cf. \cite[Proposition 1.2]{Ptak}}] \label{ptak} Let $\PMN \in \MP(\HH).$ Then
	\begin{equation*} \label{ptak0} \PMN=(I-P_\N P_\M)^{-1}P_{\N^{\perp}}|_{\M+\N}.
	\end{equation*}
\end{cor}

\begin{proof} By Proposition \ref{Grevilleprop} and \eqref{TT1},
	\begin{align*} 
		\noindent (I-P_\N P_\M)\PMN&=(I-P_\N P_\M)P_\M(P_{\N^{\perp}}P_\M)^{-1}P_{\N^{\perp}}\\
		&=P_{\N^{\perp}} P_\M(P_{\N^{\perp}}P_\M)^{-1}P_{\N^{\perp}}=I_{P_{\N^{\perp}}(\M)}P_{\N^{\perp}}=P_{\N^{\perp}}|_{\M+\N}.
	\end{align*}
	Then, multiplying  both sides of the equality $(I-P_\N P_\M)\PMN=P_{\N^{\perp}}|_{\M+\N}$  by  $(I-P_\N P_\M)^{-1}$ we get that
	$$(I \ \hat{+} \ (\{0\} \times \ker(I-P_\N P_\M)))\PMN=(I-P_\N P_\M)^{-1}P_{\N^{\perp}}|_{\M+\N}.$$ 
	But $\ker(I-P_\N P_\M)=\M\cap \N.$ In fact, if $x \in \ker(I-P_\N P_\M)$ then $x=P_\N P_\M x,$  so that $\Vert P_\M x\Vert^2 = \PI{P_\M x}{P_\M P_\N P_\M x}=\Vert P_\N P_\M x\Vert^2=\Vert x \Vert^2.$
	Hence $x \in \M \cap \N.$ Therefore $\ker(I-P_\N P_\M)\subseteq \M\cap \N \subseteq \ker(I-P_\N P_\M)$ and  $(I \ \hat{+} \ (\{0\} \times (\M \cap \N)))\PMN=(I-P_\N P_\M)^{-1}P_{\N^{\perp}}|_{\M+\N}$  or
$\PMN=(I-P_{\N}P_{\M})^{-1}P_{\N^{\perp}}|_{\M+\N}.$
\end{proof}

\section{Decompositions of multivalued projections} 	\label{sec:four}

Next, we study decomposable and continuous multivalued projections and some of their distinguished orthogonal decompositions. 

\subsection{Decomposable and continuous multivalued projections}
 
For $E:=\PMN,$ consider $P=P_{\overline{\dom} E^*}=P_{\ol{\M^{\perp}+\N^{\perp}}}$ and $E_0=E \cap (\ol{\M+\N} \times \ol{\M^{\perp}+\N^{\perp})}.$

\begin{lema} \label{E_0} The operator  $(\PMN)_0$ is a projection. More precisely, 
	$$(\PMN)_0=P_{\M \cap \ol{\M^{\perp}+\N^{\perp}} \pl \N}.$$
\end{lema}
\begin{proof} By definition, $E_0=E \cap (\ol{\M+\N} \times \ol{\M^{\perp}+\N^{\perp}})=\{ (m+n,m): m \in \M \cap  \ol{\M^{\perp}+\N^{\perp}}, n \in \N\}=P_{\M \cap \ol{\M^{\perp}+\N^{\perp}} \pl \N}.$
\end{proof}
	
\begin{prop} \label{propdecom} The following conditions are equivalent:
\begin{enumerate} 
	\item $P_{\M, \N}$ is decomposable;
	\item $\M+\N=\M\cap \overline{\M^\bot+\N^\bot}\dotplus\N;$
	\item  $P_{\ol{\M} \cap \ol{\N}}(\M) = \M \cap \N,$
	\item $\M=\M\cap \overline{\M^\bot+\N^\bot}\oplus \M\cap\N.$
	\end{enumerate}
If any of the above conditions holds $\ol{\M \cap \N}=\ol{\M}\cap \ol{\N}$ and
$$P_{\M, \N}=P_{\M\cap (\M\cap N)^\bot \pl \N}\ \hat{\oplus} \ (\{0\} \times \M \cap \N).$$
\end{prop}
	
\begin{proof} $i)\Leftrightarrow ii)$ By definition, $P_{\M, \N}$ is decomposable if and only if $\M+\N=\dom E=\dom E_0=\M\cap \overline{\M^\bot+\N^\bot} \dotplus \N.$

$i)\Leftrightarrow iii)$ Theorem \ref{decomthm}  establishes that $P_{\M, \N}$ is decomposable if and only if $\ran E_{\sing} \subseteq \mul E.$ But $\ran E_{\sing}=\ran(I-P)E=P_{\ol{\M}\cap\ol{\N}}(\M).$ 

$iii)\Leftrightarrow iv)$ is straightforward.

If any of the above conditions holds then $E=E_0 \ \hat{\oplus} \ E_{\mul}$ and, by Theorem \ref{decomthm}, $\ol{\M \cap \N}=\ol{\M}\cap \ol{\N}.$  Then $E_0=P_{\M\cap (\M\cap N)^\bot \pl \N}.$
\end{proof}

\begin{cor}\label{PyI-P} If $\M, \N$ are closed subspaces of $\HH$ then
$$P_{\M, \N}=P_{\M \ominus (\M\cap \N) \pl \N} \ \hat{\oplus} \ (\{0\} \times \M \cap \N).$$
\end{cor}

\begin{proof} If $\M$ and $\N$ are closed then $P_{\M,\N}$ is closed and therefore decomposable \cite[Corollary 3.15]{Hassi2}. 
\end{proof}

The fact that $\PMN$ is decomposable does not imply that $P_{\N,\M}$ is decomposable as the next example shows. 

\begin{Example} Consider $\M, \N$ proper subspaces of $\HH$ such that  $\M\subsetneq \N$ and $\M$ is dense. For example take $A \in L(\HH)$ positive (semi-definite) such that $A$ is not invertible and set $\M:=\ran A$ and $\N:=\ran A^{1/2}.$  Then $P_{\M,\N}$ is decomposable but $P_{\N,\M}$ is not  decomposable. In fact, it is clear that $\ol{\M} \cap \ol{\N}=\HH$ so that $P_{\ol{\M} \cap \ol{\N}}(\M)=\M = \M \cap \N.$ Then, by Proposition \ref{propdecom}, $P_{\M,\N}$ is decomposable. On the other hand, $ \M \cap \N \subsetneq \N =P_{\ol{\M} \cap \ol{\N}}(\N)$ then $P_{\N,\M}$ is not decomposable by Proposition \ref{propdecom}.
\end{Example}

\begin{prop} \label{decomPNM} The following conditions are equivalent:
\begin{enumerate}
	\item $\ol{\M}\cap \ol{\N}=\M \cap \N;$
	\item $\PMN$ is decomposable and $\M \cap \N$ is closed;
	\item $P_{\N,\M}$ is decomposable and $\M \cap \N$ is closed;
\end{enumerate}
\end{prop}

\begin{proof}
$i) \Leftrightarrow ii):$ If $i)$ holds then $\M \cap \N$ is closed and $P_{\ol{\M} \cap \ol{\N}}(\M)=P_{\M \cap \N}(\M)=\M \cap \N.$ By Proposition \ref{propdecom}, $\PMN$ is decomposable. See also \cite[Corollary 3.14]{Hassi2}. Conversely, suppose that $ii)$ holds. Then, by Theorem \ref{decomthm}, $\mul P_{\ol{\M}, \ol{\N}}=\clmul \PMN,$ or $\ol{\M}\cap \ol{\N}=\ol{\M \cap \N}=\M \cap \N.$
\end{proof}
	
	\begin{cor}
	Suppose that $P_{\M, \N}$ is decomposable. If $\M$ is closed, then $P_{\N, \M}$ is decomposable.
	\end{cor}
	\begin{proof}  Suppose that $\PMN$ is decomposable. Then, by Proposition \ref{propdecom}, $\ol{\M\cap \N}=\ol{\M}\cap\ol{\N}.$ Since $\M$ is closed, from $iv)$ of Proposition \ref{propdecom}, $\M \cap \N$ is closed. Then  $\ol{\M} \cap \ol{\N}=\M \cap \N$ and, by Proposition \ref{decomPNM}, $P_{\N, \M}$ is decomposable.
	\end{proof}

In \cite[Corollary 3.7]{Cross}, it is shown that $\PMN$ is continuous if and only if $\M^\bot+\N^\bot=(\M\cap\N)^\bot.$ 

\begin{cor} \label{corclosed} Suppose that $\M, \N$ are operator ranges of $\HH$ such that $\M+\N$ and $\M \cap \N$ are closed. Then $\PMN$ is decomposable and continuous. 
\end{cor}
\begin{proof} Since $\M+\N$ is closed, by Proposition \ref{ORLabrousse}, $\ol{\M}\cap \ol{\N}=\ol{\M\cap \N}=\M \cap \N.$ Then, by Proposition \ref{decomPNM},  $\PMN$ is decomposable. Also, by Proposition \ref{ORLabrousse}, $(\M \cap \N)^{\perp}=\M^{\perp}+\N^{\perp}$ so that $\PMN$ is continuous. 
\end{proof} 

The condition ``$\M^\bot+\N^\bot$ closed"  does not imply the continuity of $\PMN$ \cite[Example 3.10]{Cross}. However, if $\PMN$ is decomposable, this conditions is indeed sufficient.

\begin{prop} \label{PropDC} Suppose that $\PMN$ is decomposable. Then $\PMN$ is continuous if and only if $\M^{\perp}+\N^{\perp}$ is closed.
\end{prop}

\begin{proof} $\PMN$ is continuous if and only if $(\PMN)_m$ is bounded. But $(\PMN)_m=(\PMN)_{\reg}$, by Theorem \ref{decomthm}; because $\PMN$ is decomposable. By Proposition \ref{bounded3},  $(\PMN)_{\reg}$ is bounded if and only if $\dom \PMN^*=\M^{\perp}+\N^{\perp}$ is closed. 
\end{proof}

\begin{prop} \label{PropDC1 }Suppose that that $\M \cap \N$ is closed. Then the following are equivalent:
\begin{enumerate}
	\item $\PMN$ is continuous;
	\item $\PMN$ is decomposable with bounded operator part;
	\item $\ol{\M} \cap \ol{\N}=\M \cap \N$ and $\M^{\perp}+\N^{\perp}$ is closed.
\end{enumerate} 
\end{prop}
\begin{proof} $\noindent i) \Leftrightarrow ii)$: follows from \cite[Corollary 3.22]{Hassi2}.

$\noindent ii) \Leftrightarrow iii)$: if $ii)$ holds then, by Proposition \ref{propdecom}, $\ol{\M} \cap \ol{\N}=\ol{\M \cap \N}=\M\cap \N,$ and from Proposition \ref{PropDC} we get that $\M^{\perp}+\N^{\perp}$ is closed. The converse follows from Proposition \ref{decomPNM} and Proposition \ref{PropDC}.
\end{proof}

\begin{obs} For $\M$ and $\N$  closed subspaces of $\HH,$ Friedrichs  \cite{Friedrichs} defined the cosine of the \emph{angle} between $\M$ and $\N$ as 
	\begin{align*}
		c(\M,\N):=\sup \left\{\vert \! \PI{x}{y} \!\vert \! \! : x\in \M\ominus(\M \cap \N), \! y\in\N \ominus(\M \cap \N),  \|x\|\ \!, \! \|y\|\leq 1 \right\}.
	\end{align*}
By \cite[Theorem 13]{Deutsch}, $\M+\N$ is closed if and only if $c(\M,\N)<1,$ and by \cite[Lemma 11]{Deutsch}, $\M+\N$ is closed if and only if $\M^\bot+\N^\bot$ is closed. 

Therefore, if $\PMN$ is  closed, i.e., if  $\M$ and $\N$ are closed, then, by Proposition \ref{Tcont}, $\PMN$ is continuous if and only if $\M+\N$ is closed, or equivalently $c(\M,\N)<1.$ Now, if we replace the condition of $\PMN$ being closed by the less restrictive condition of $\PMN$ being decomposable then, by Proposition \ref{PropDC}, $\PMN$ is continuous if and only if $c(\ol{\M}, \ol{\N})<1.$ 
\end{obs}

\subsection{Distinguished orthogonal range decompositions of  multivalued projections}
When a multivalued projection $E$ is decomposable, $E_m=E_{\reg}=E_0$ and then, by Lemma \ref{E_0}, $E_m$ and $E_{\reg}$ are projections. More generally, we are interested in characterizing the multivalued projections whose operator parts in their Lebesgue decompositions are projections. 

\begin{lema} \label{lemaF} For a given $F \in \mc P,$ $F\PMN$ is a projection if and only if $\M \cap \N \subseteq \ker F,$
$\M+\N = F(\M)+\M\cap \ker F+\N$ and $F(\M)\subseteq \M+\N \cap \ker F.$
In this case,
$$F\PMN=P_{F(\M) \pl \N+\M \cap\ker F}.$$
\end{lema}

\begin{proof} Set $E:=F\PMN.$ It is easy to check that $\dom E=\M+\N,$ $\ran E=F(\M),$ $\ker E=\N+\M \cap \ker F$ and $\mul E=F(\M \cap \N).$ 
	
Suppose that $E$ is a projection. Then $\mul E=\{0\},$ or $\M \cap \N \subseteq \ker F.$ Also, $\dom E=\ran E+\ker E$ so that $\M+\N = F(\M) \dotplus (\M\cap \ker F+\N).$ Finally, if $x \in F(\M),$ write  $x=m+n$ with $m\in \M$ and $n \in \N.$ Then $x=Ex=F\PMN(m+n)=Fm.$  So that $m+n=Fm$ and $n=-(I-F)m \in \ker F,$ or  $x \in \M+\N \cap \ker F.$

Conversely, since $\M \cap \N \subseteq \ker F,$ $E$ is an operator. Let us see that $\ran E \cap \ker E=\{0\}.$ Let $m \in \M$ and suppose that $Fm=n+u,$ $n \in \N$ and $u\in \M \cap \ker F.$ Then $n=Fm-u \in F(\M)+\M \cap \ker F\subseteq \M + \N \cap \ker F.$ Then $n=m'+v,$ $m' \in \M$ and $v \in \N \cap \ker F.$ So that $m'=n-v\in \M \cap \N \subseteq \ker F.$ Hence $Fm=m'+v+u \in \ker F$ and $Fm=0.$ Then we can consider $P_{F(\M) \pl \N+\M \cap\ker F}.$ By assumption, $\dom E= \M+\N= F(\M)+\M\cap \ker F+\N=\dom P_{F(\M) \pl \N+\M \cap\ker F}.$ Finally, take $x \in F(\M)$  and write $x=m+n$ with $m\in \M$ and $n \in \N \cap  \ker F.$ Then $x=Fx=Fm=F\PMN (m+n)=F\PMN x$ because $(x,m)=(m+n,m)\in \PMN$ and $(m,Fm) \in F.$ Therefore $P_{F(\M) \pl \N+\M \cap\ker F} \subseteq E$ and equality follows. 
\end{proof}

 Recall that for $T:=\PMN$, $T_{\reg}=PT$ and $T_m=QT$ where $P=P_{\overline{\M^{\perp}+\N^{\perp}}}$ and $Q=P_{(\M \cap \N)^{\perp}}.$
\begin{cor} \label{propreg} The following statements hold:
\begin{enumerate}
	\item[1.] $(\PMN)_{\reg}$ is a projection if and only if $\M+\N=P(\M)+\M\cap \ol{\N}+\N$ and $P(\M) \subseteq \M + \N \cap \ol{\M}.$ 
	\item[2.] $(\PMN)_{m}$ is a projection if and only if $\M+\N=Q(\M)+\M\cap \ol{\M \cap \N}+\N$ and $Q(\M) \subseteq \M +\N \cap  \ol{\M \cap \N}.$
\end{enumerate}

\end{cor}
\begin{proof} It holds that $\M \cap \N \subseteq \ker P=\ol{\M} \cap \ol{\N}$  and $\M \cap \N \subseteq \ker Q=\ol{\M\cap \N}$ then the results in $1$ and $2$ follow from Lemma \ref{lemaF}.
\end{proof}

	\section{Semiclosed multivalued projections} \label{sec:five}
	
	The definitions of \emph{semiclosed subspace} and \emph{semiclosed operator} were formally introduced by Kaufman \cite{Kaufman}, though these notions were considered by other authors before.
	
	\begin{Def}
		A subspace $\St$ of $\HH$ is \emph{semiclosed} if there exists an inner product $\PI{\cdot}{\cdot}^{'}$ such that $(\St,\PI{\cdot}{\cdot}^{'})$ is a Hilbert space which is \emph{continuously imbedded} in $\HH,$ i.e., there exists $b > 0$ such that $\PI{x}{x} \leq b \PI{x}{x}'$ for every $x \in \St.$ 
	\end{Def}

	A subspace $\St$ is semiclosed if and only if $\St$ is an operator range: in fact, if $T \in L(\HH)$ define
	$$\PI{u}{v}_T :=\PI{T^{\dagger} u}{T^{\dagger} v}\mbox{ for } u, v \in \ran(T),$$ where $T^{\dagger}$ denotes the (possibly unbounded) Moore-Penrose inverse of $T$ \cite{Nashed}, and let $\Vert\cdot\Vert_T$ be the induced norm. Then $(\ran T, \PI{\cdot}{\cdot}_T)$ is a Hilbert space and
	\begin{equation} \label{eqT}
		\Vert u \Vert = \Vert TT^{\dagger} u \Vert \leq \Vert T \Vert \Vert T^{\dagger} u \Vert =\Vert T \Vert \Vert  u \Vert_T, \mbox{  for } u \in \ran(T).
	\end{equation}
	Therefore, $\ran T$ is semiclosed. We write
	$$\mc{M}(T) := (\ran T, \PI{\cdot}{\cdot}_T)$$ to denote the space $\ran T$ equipped with the Hilbert space structure $\PI{\cdot}{\cdot}_T.$
	
Conversely, if $\St$ is a semiclosed subspace of $\HH,$ then there is a unique positive (semi-definite) operator $T \in L(\HH)$ such that $(\St, \PI{\cdot}{\cdot}_T)=\M(T),$ see  \cite[Corollary 3.3]{AndoSC} and \cite[Theorem 1.1]{Filmore}. In what follows, we use the terms operator range and semiclosed interchangeable.

\begin{thm}[{\cite[Corollary 3.8]{AndoSC}}] \label{thmAndo} For $T_1, T_2 \in L(\HH),$ let $T:=(T_1{T_1}^*+T_2{T_2}^*)^{1/2}.$ Then $\Vert u_1 + u_2 \Vert_T^2 \leq \Vert u_1 \Vert_{T_1}^2 + \Vert u_2 \Vert_{T_2}^2,$ for $u_1 \in \ran T_1$ and $u_2 \in \ran T_2,$ and for any $u \in \ran T,$ there are unique  $u_1 \in \ran T_1$ and $u_2 \in \ran T_2$ such that $u=u_1+u_2$ and $$ \Vert u_1 + u_2 \Vert_T^2 = \Vert u_1 \Vert_{T_1}^2 + \Vert u_2 \Vert_{T_2}^2.$$
\end{thm}

The Hilbert space $\mc{M}(T)$ plays a significant role in many areas, in particular in the de Branges complementation theory, see \cite{AndoSC, Bolotnikov}.

	\begin{Example} \label{Example3} If $T$ is a contraction, the inclusion map $\iota: \mc{M}(T) \ra \HH$ is contractive, since $\Vert \iota x \Vert \leq \Vert x \Vert_T$ for every $x \in \ran T.$ Then, we say that $\St:=\M(T)$ is \emph{contractively included} in $\HH.$ In this case, the \emph{de Branges-Rovnyak complement} of $\St$ is defined by 
	$$\St':=\M((I-TT^*)^{1/2}).$$
In fact,	$\St+\St'=\ran ((TT^*+(I-TT^*))^{1/2})=\HH$ and the \emph{overlapping space} $\St \cap \St'$ measures the extent to which this complementary space fails to be a true orthogonal complement, see \cite[Proposition 3.4]{Bolotnikov}. 

In this case, $P_{\St, \St'}$ defines a multivalued projection with $\dom P_{\St, \St'}=\St+\St'=\HH,$ and $\mul P_{\St, \St'}=\St \cap \St'$. Moreover, $P_{\St, \St'}$ is a contraction, in the sense that $\Vert  P_{\St, \St'}\Vert \leq 1.$
	In fact,  since $I=(T_1{T_1}^*+T_2{T_2}^*)^{1/2}$ for $T_1:=T$ and $T_2:=(I-TT^*)^{1/2},$ by Theorem \ref{thmAndo}, given any $x \in \HH,$ there are unique $f \in \St$ and $g\in \St'$ such that $x=f+g$ and $$\Vert x \Vert^2=\Vert f \Vert_{T_1}^2+\Vert g \Vert_{T_2}^2.$$ 
	If $Q:=P_{(\St \cap\St')^{\perp}}$ then $(x,Qf) \in  QP_{\St, \St'}.$ Hence
	\begin{align*}
		\Vert (QP_{\St, \St'})x\Vert^2&=\Vert Qf\Vert^2 \leq \Vert f \Vert^2 \leq \Vert f \Vert_{T_1}^2 \leq \Vert f \Vert_{T_1}^2+\Vert g \Vert_{T_2}^2=\Vert x \Vert^2.
	\end{align*}	So that $\Vert P_{\St, \St'} \Vert=	\Vert QP_{\St, \St'} \Vert \leq 1.$
\end{Example}

	In \cite{Hassi} operator range linear relations are considered. See also \cite{Kaufman,Koliha}.

	\begin{Def} A linear relation $T \in \lr(\mc E,\KK)$ is an \emph{operator range} (or a \emph{semiclosed relation}) if it is a semiclosed subspace of $\mc E \times \KK,$ that is, $T=\ran \Phi$ for some $\Phi \in L(\mc H, \mc E \times \KK)$ where $\mc H$ is a Hilbert space.
	\end{Def}
Given $A \in L( \mc E, \mathcal H)$ and $B \in L(\mathcal K, \HH),$ consider the row operator $$\begin{bmatrix}
	A & B
\end{bmatrix} \in L(\mc E \times \KK, \mc H), \  \begin{bmatrix} A & B
\end{bmatrix} \begin{bmatrix}
	h \\ k
\end{bmatrix}=Ah+Bk \mbox{ for } h \in \mc E, k \in \mc K;$$ and 
for a pair of operators $C \in L(\HH, \mathcal E)$ and $D \in L(\mathcal H, \KK),$ consider the column operator
$$\begin{bmatrix} C \\ D
\end{bmatrix}\in L(\mathcal H,\mathcal E\times \KK ), \  \begin{bmatrix} C \\ D
\end{bmatrix}x=\begin{bmatrix} Cx \\ Dx
\end{bmatrix} \mbox{ for } x \in \mc H.$$ 

Set \begin{equation} \label{Gamma}
	\G:=\left( \begin{bmatrix}
		A & B
	\end{bmatrix} \begin{bmatrix}
	A & B
\end{bmatrix}^*\right)^{1/2}=(AA^*+BB^*)^{1/2}.
\end{equation}  
So $\G \in L(\HH)$ and  $\ran \G=\ran \begin{bmatrix} A & B
\end{bmatrix}=\ran A+\ran B.$ 
By Douglas' Lemma \cite{Douglas}, there exist (unique) contractions $C_A \in L(\mc E, \HH)$  and $C_B \in L(\KK, \mc H)$ such that 
\begin{equation} \label{Douglas1}
	A=\G C_A \ \ \mbox{ and } \ \ B=\G C_B,
\end{equation}
and $\ran C_A, \ran C_B \subseteq \clran \G.$ Moreover, the identity 
\begin{equation} \label{DPolar}
	\begin{bmatrix} A & B
	\end{bmatrix} =\G \begin{bmatrix} C_A & C_B
	\end{bmatrix} 
\end{equation}
is the  (left) polar decomposition of $\begin{bmatrix} A & B
\end{bmatrix},$ where the row operator  $\begin{bmatrix} C_A & C_B
\end{bmatrix}$ is the (unique) partial isometry with
\begin{equation} \label{pisometry}
	\ran \begin{bmatrix} C_A & C_B
	\end{bmatrix} = \clran \G
	\ \textrm{and} \
	\ker\begin{bmatrix} C_A & C_B
	\end{bmatrix}=\ker\begin{bmatrix} A & B
	\end{bmatrix},
\end{equation} 
\cite[Lemma 4.2]{Hassi}. In particular,
\begin{equation}\label{P_G}
	P_{\clran \G}=C_A^{\phantom{*}}C_A^*+C_B^{\phantom{*}}C_B^*,
\end{equation}
and 
\begin{equation}\label{P_G2}
	\G=AC_A^*+BC_B^*.
\end{equation}

Following the notation used in \cite{Hassi} write
$$L(C,D):= \ran  \begin{bmatrix} C \\ D
\end{bmatrix}=\{(Cx,Dx) : x \in \mathcal H\}.$$
Thus, $L(C,D)$ is a semiclosed relation. Since \begin{equation} \label{cociente}
	L(C,D)=DC^{-1},
\end{equation} 
in the sense of the product of linear relations, $L(C,D)$ is a \emph{quotient}.
In fact, any semiclosed linear relation is a quotient as (\ref{cociente}), that is, if $T\in\lr(\mc E,\KK)$ is semiclosed then  there exist $C\in L(\mathcal H, \mc E)$ and $D\in L(\mathcal H, \KK)$ such that $T=L(C,D)$ \cite[Theorem 1.10.1]{BHS}. It is straightforward that $\dom L(C,D)=\ran C,$  $\ran L(C,D)=\ran D$, $\ker L(C,D)=C(\ker D)$ and $\mul L(C,D)= D(\ker C).$ Hence, if $T$ is an operator range relation then $\ran T, \dom T, \ker T$ and $\mul T$ are semiclosed subspaces, \cite[Corollary 2.9]{Hassi}.

A projection $E$ is an operator range operator if and only if $\ran E$ and $\ker E$ are operator ranges, \cite[Proposition 3.2]{Semiclosed}. More generally,

\begin{prop} \label{propSC1} $E \in \Mp(\HH)$ is an operator range relation if and only if $\ran E$ and $\ker E$ are semiclosed subspaces.
\end{prop}

\begin{proof} Let  $E \in \Mp(\HH)$ and suppose that $\ran E= \ran A$ and $\ker E=\ran B,$ for $A, B \in L(\HH).$ Then 
$E=L\left( \begin{bmatrix} A & B
\end{bmatrix}, \begin{bmatrix}
A & 0
\end{bmatrix}\right)$, i.e., $E$ is an operator range relation.
The converse follows from the discussion above.
\end{proof}

\begin{Example} \label{Example4} Given $W\in L(\HH)$ positive (semi-definite) consider the semi-inner product
$\PI{x}{y}_W:=\PI{Wx}{y} \mbox{ for } x, y \in \HH$ and  the semi-norm
$\Vert x \Vert_W:=\PI{x}{x}_W^{1/2}=\Vert W^{1/2} x\Vert, \ x \in \HH.$ Given a subspace $\St \subseteq \HH,$ the $W$-\emph{orthogonal companion} of $\St$ is the (closed) subspace
$\St^{\perp_W}:=\{x \in \HH: \PI{x}{s}_W=0 \mbox{ for every } s \in 
\St\}.$
	
Let $A\in L(\HH)$ and consider the multivalued projection 
$$P_{W,\ran A}:=P_{\ran A, \ \ran A^{\perp_W}}.$$ Notice that $P_{W,\ran A}$ is semiclosed since both the range and the kernel are semiclosed subspaces. This multivalued projection plays a fundamental role 
in the study of some approximation problems, see \cite{Matrix}. For instance, given  $b\in\HH$ a vector $x_0\in\HH$ is a $W$-\emph{least squares solution} ($W$-LSS) of $Ax=b$ if 
\begin{equation}\label{ALSS}
||Ax_0-b||_{W}= \underset{y \in \ran A}{\min} \ \ ||y-b||_{W}.
\end{equation}
In \cite[Propositions 5.1 and 5.3]{Matrix}, it is proven that there is a solution of \eqref{ALSS} if and only if $b \in \dom P_{W, \ran A}$, and  $A^{-1}P_{W,\ran A}b$ is the set of $W$-LSS of $Ax=b$.

\end{Example}

\begin{prop} \label{Ando3} Given  $A, B \in L(\HH),$ consider $\M=\ran A$ and $\N=\ran B$ and $C_A, C_B$ and $ \G$ as in (\ref{Douglas1}) and (\ref{Gamma}). Then 
\begin{equation}
				\PMN=\G C_A^{\phantom{*}}C_A^*\G^{-1} \ \hat{\dotplus} \ (\{0\} \times (\M \cap \N)),
			\end{equation}
		where $\G C_A^{\phantom{*}}C_A^*\G^{-1}$ is an operator part of $\PMN.$
		\end{prop}
		
		\begin{proof} From \eqref{P_G2} we get that
		\begin{equation} \label{eqlema43}
		\PMN \G=\PMN(AC_A^*+BC_B^*)=AC_A^* \ \hat{\dotplus} \ (\{0\} \times (\M \cap \N)).
		\end{equation}
	In fact, $(x,y) \in \PMN \G$ if and only if $(\G x,y) =(AC_A^*x+BC_B^*x,y) \in \PMN$ if and only if $y=AC_A^*x+w$ for some $w \in \M \cap \N,$ if and only if $(x,y) \in AC_A^* \ \hat{\dotplus} \ (\{0\} \times (\M \cap \N)).$
	Then $\PMN=\PMN \G \G^{-1}= AC_A^*\G^{-1}  \hat{+} (\{0\} \times (\M \cap \N)).$ Finally, $\mul (\G C_A^{\phantom{*}}C_A^*\G^{-1})=\G C_A^{\phantom{*}}C_A^*(\ker \G)=\{0\}$ and then the sum is direct. 
\end{proof}

As a corollary, we get the following formula, similar to the one in \cite[Theorem 2.2]{Ando} and
 \cite[Proposition 3.3]{Semiclosed} for operators. 
 
\begin{cor} \label{propSC12} Given $A, B \in L(\HH), $ consider $\M=\ran A$ and $\N=\ran B.$ Then
	$$\PMN=(\G^{-1}AA^*)^*\G^{-1} \ \hat{\dotplus} \ (\{0\} \times \M \cap \N),$$ where  $\G$ is as in \eqref{Gamma}. 
\end{cor}
\begin{proof} By Proposition \ref{Ando3},
	\begin{align*}
		\PMN&=\G C_A^{\phantom{*}}C_A^*|_{\clran \G}\G^{-1} \ \hat{\dotplus} \ (\{0\} \times \M \cap \N),
	\end{align*}
	where we use that $\G C_A^{\phantom{*}}C_A^* \G^{-1}= AC_A^*|_{\clran \G}\G^{-1}$ because  $AC_A^*|_{\clran \G}\G^{-1}\subseteq AC_A^*\G^{-1},$ $\dom (AC_A^*|_{\clran \G}\G^{-1})= \dom (AC_A^*\G^{-1})=\M+\N$ and $\mul (AC_A^*\G^{-1})=AC_A^*(\ker \G)=\{0\}.$ But $\G C_A^{\phantom{*}}C_A^*|_{\clran \G}=(\G^{-1}AA^*)^*.$ In fact, from $AA^*=\G C_A^{\phantom{*}}C_A^*\G$ we get that $\G^{-1}AA^*=(I \ \hat{+} \ (\{0\} \times \ker \G))C_A^{\phantom{*}}C_A^*\G.$ Then
$(\G^{-1}AA^*)^*=(C_A^{\phantom{*}}C_A^*\G \ \hat{\dotplus} \ (\{0\} \times \ker \G))^*=\G C_A^{\phantom{*}}C_A^* \cap (\clran \G \times \HH)=\G C_A^{\phantom{*}}C_A^*|_{\clran \G}.$
\end{proof}

In \cite[Theorem 2.3]{Ando}, Ando proved that if $P_{\M \pl \N}$ is a  densely defined closed projection and $\G := (P_\M + P_\N)^{1/2}$ then the operator $P_0:=\G^{-1}P_{\M \pl \N}\G$ is well defined and it is an orthogonal projection. A bounded operator which is injective with dense range is called a \emph{quasi-affinity}. An operator $T$ is quasi-affine to $C$ if there is a quasi-affinity $X$ such that $TX = XC.$ 
In these terms, $P_{\M \pl \N}$ is quasi-affine to $P_0$, or $$P_{\M \pl \N}\G=\G P_0.$$ Analogous results can be obtained for semiclosed multivalued projections. 

\begin{lema}\label{D}
	Let $\G\in \lr(\HH)$ such that $\ran \G=\M+\N$ and $\mul\G\subseteq \N.$ Then
	$$\G^{-1}\PMN\G=P_{\G^{-1}(\M), \G^{-1}(\N)}.$$
\end{lema}
\begin{proof}
	If $E:=\G^{-1}\PMN\G$ then, by (\ref{TT1}), we get that
	$$E^2=\G^{-1}\PMN(I_{\ran \G}\hat{+}(\{0\}\times \mul\G))\PMN\G=E,$$
	$\dom E=\dom \G$ and $\ran E=\G^{-1}(\M)\subseteq \dom \G$. Then, $E\in\Sp(\HH).$ Finally,
	$I-E=P_{\ker E, \ran E}$ implies that $\ker E=\ran (I-E)=\G^{-1}(\N).$
\end{proof}

The multivalued projection $\PMN$ is quasi-affine to a multivalued projection with domain $\HH$, having  a positive semidefinite contraction as an operator part, in the sense that there exists a positive bounded operator $\G$ with $\ran \G= \M+\N$ such that $\G$ intertwines  $\PMN$ and this multivalued projection.

\begin{cor} Given  $\M$ and $\N$ operator ranges. Then 
$$\PMN X=X ( C  \ \hat{\dotplus} \ (\{0\} \times \St)),$$
where $X,C\in L(\overline{\M+\N})$ are positive, $X$ is a quasi-affinity, $C$ is a contraction and $\St$ is an operator range. 
\end{cor}
\begin{proof}
From \eqref{eqlema43}, $\G^{-1} (\PMN \G)=(I \hat{+} (\{0\} \times \ker \G))C_A^{\phantom{*}}C_A^*  \hat{+} \ (\{0\} \times \G^{-1}(\M \cap \N))=C_A^{\phantom{*}}C_A^* \ \hat{+} \ (\ker C_A^* \times \ker \G) \ \hat{+} \ (\{0\} \times \G^{-1}(\M \cap \N))=C_A^{\phantom{*}}C_A^* \ \hat{\dotplus} \ (\{0\} \times \G^{-1}(\M \cap \N)).$ Take $X:=\G|_{\overline{\M+\N}},$ $C:=C_AC_A^*|_{\overline{\M+\N}}$ and $\St:=\ran C_A\cap \ran C_B.$
\end{proof}

\begin{lema} \label{MNclosed0} Given  $A, B \in L(\HH),$ consider $\M=\ran A$ and $\N=\ran B.$ Then $\G^{-1}(\M \cap \N)$ is closed if and only if $$\M+\N=\G(\ker C_B^*) \dotplus \G(\ker C_A^*) \dotplus \M \cap \N,$$ where $C_A, C_B$ and $ \G$ are as in \eqref{Douglas1} and \eqref{Gamma}. In this case,
\begin{equation} \label{eqpropcerrado}
\M=\G(\ker C_B^*) \dotplus \M \cap \N \mbox { and } \N=\G(\ker C_A^*) \dotplus \M \cap \N.
\end{equation}
\end{lema}

\begin{proof} By \eqref{pisometry}, $\ker \G= \ker C_A^* \cap \ker C_B^*,$ and $\G^{-1}(\M \cap \N)=\ran C_A \cap \ran C_B \oplus \ker C_A^* \cap \ker C_B^*.$ Then 
\begin{equation}\label{GMN}
\G^{-1}(\M \cap \N) \textrm{ is closed if and only if} \ran C_A \cap \ran C_B \textrm{ is closed}.	\end{equation}
	
Suppose that $\G^{-1}(\M \cap \N)$ is closed. Then, from \eqref{GMN} and Proposition \ref{ORLabrousse} it follows that $\clran C_A \cap \clran C_B=\ol{\ran C_A \cap \ran C_B}=\ran C_A \cap \ran C_B.$ Also, by \eqref{pisometry} , $\ker C_A^* + \ker C_B^*$ is closed. Then
$$\HH=(\ker C_A^* + \ker C_B^*) \oplus \ran C_A \cap \ran C_B.$$ 
Applying $\G$ to both sides of the last equality it follows that
\begin{equation}\label{eqM+N}
\M+\N=\G(\ker C_B^*) \dotplus \G(\ker C_A^*) \dotplus \M \cap \N,
\end{equation} where the sums are direct. Indeed, if $\G x=\G y$ for some $x\in \ker C_A^* $ and 
$y\in \ker C_B^*$, then $x-y \in \ker \G=\ker C_A^* \cap \ker C_B^*$ so that $x, y \in \ker C_A^* \cap \ker C_B^*$ or  $\G x=\G y=0.$  On the other hand, if $\G x=\G z$ for some $x\in \ker C_A^*+ \ker C_B^*$ and $z\in \ran C_A\cap \ran C_B  \subseteq  \clran \G$ then $x-z \in \ker \G.$ Thus, $x=z+s$ for some $s\in \ker \G,$ and so $z=x-s\in \clran \G \cap \ker\G=\{0\}$ and $\G x=\G z=0.$

Conversely, if $\M+\N=\G(\ker C_B^*) \dotplus \G(\ker C_A^*) \dotplus \M \cap \N$, applying $\G^{-1}$ to both sides of this equation, we get that
$\HH=(\ker C_B^* + \ker C_A^*) \oplus \ran C_A \cap \ran C_B.$ Thus $\ran C_A \cap \ran C_B$ is closed, or equivalently $\G^{-1}(\M\cap \N)$ is closed.

In this case \eqref{eqM+N} implies that 
$$\M=\G(\ker C_B^*) \dotplus \M \cap \N.$$

Indeed, the inclusion $\supseteq$ always holds, since $\G(\ker C_B^*)=AC_A^*(\ker C_B^*)\subseteq \M.$ To see the reverse inclusion take $x \in \M$ and write $x=x_1+x_2+x_3$ with $x_1 \in \G(\ker C_B^*),$ $x_2 \in \G(\ker C_A^*)$ and $x_3 \in \M \cap\N.$ On the one hand, $x_2 \in\N$ and, on the other hand, $x_2=x-x_1-x_3 \in \M.$ 
Then $x_2\in \M \cap \N \cap \G(\ker C_A^*)=\{0\}$ so that $x=x_1+x_3 \in \G(\ker C_B^*) \dotplus \M \cap \N.$
In a similar fashion, $\N=\G(\ker C_A^*) \dotplus \M \cap \N.$
\end{proof}

\begin{cor}\label{MNclosed}
Given  $A, B \in L(\HH),$ consider $\M=\ran A$ and $\N=\ran B$ such that $\M \cap \N$ is closed in $\M+\N.$ Then 
$$\PMN=P_{\G(\ker C_B^*)\pl \N}\hat{\dotplus}\ (\{0\} \times (\M \cap \N)).$$ 	
\end{cor}
\begin{proof} If $\M \cap \N$ is closed in $\M+\N$ then $\G^{-1}(\M\cap\N)=\G^{-1}(\ol{\M\cap \N})$ is closed. By Lemma \ref{MNclosed0} and (\ref{eqM+N}), $P_{\G(\ker C_B^*)\pl \N}$ is well defined. Set $E:=P_{\G(\ker C_B^*)\pl \N}\hat{\dotplus}\ (\{0\} \times (\M \cap \N)),$ then $\mul E =\M \cap \N=\mul \PMN$ and, by Lemma \ref{MNclosed0}, $\dom E=\M+\N=\dom \PMN.$ Clearly, $E \subseteq \PMN$ because $\G(\ker C_B^*) \subseteq \M,$ and the result follows. 
\end{proof}

\begin{thm}\label{propP0}
	Let $E\in \Sp(\HH)$ with $\ran E=\ran A$ and $\ker E=\ran B$ for some $A,B\in L(\HH).$ Suppose that $\G^{-1}(\mul E)$ is closed where $\G$ is as in (\ref{Gamma}). Then $\G^{-1}E\G \in \MP(\HH)$ with $(\G^{-1}E\G)_0 \in \mc P.$
	
	Conversely, if $E\in \lr(\HH)$ and there exists $\G\in L(\HH)$ positive (semi-definite) with $\ran \G=\dom E$ such that $\G^{-1}E\G \in \MP(\HH)$ with $(\G^{-1}E\G)_0 \in \mc P$ then $E$ is a semiclosed multivalued projection.
\end{thm}
\begin{proof} Write $\M=\ran A$ and $\N=\ran B,$ and suppose that $\G^{-1}(\M\cap\N)$ is closed, or equivalently, by (\ref{GMN}), $\ran C_A\cap\ran C_B$ is closed. It follows from \eqref{eqpropcerrado} that \begin{equation}\label{B}
\G^{-1}(\M)=\ker C_B^*\oplus \ran C_A\cap\ran C_B,
\end{equation} 
so that $\G^{-1}(\M)$ is closed. On the other hand, $\M=\ran \G C_A$ and then
\begin{equation}\label{C}
	\G^{-1}(\M)=\ran C_A\oplus \ker C_A^*\cap\ker C_B^*,
	\end{equation}
so $\ran C_A$ is closed. The analogous to (\ref{B}) and (\ref{C}) hold for $\G^{-1}(\N)$ so that $\ran C_B$ is closed and 
\begin{equation} \label{V}
\G^{-1}(\N)=\ran C_B \oplus \ker C_A^* \cap \ker C_B^*.
\end{equation}

Then, by Lemma \ref{D}, 
$$\G^{-1}\PMN \G = P_{\G^{-1}(\M), \G^{-1}(\N)}=P_{\ker C_B^*}\hat{\dotplus}(\{0\}\times \G^{-1}(\M \cap \N)).$$
To see the last equality, notice that the last two relations have the same domain, $\HH,$ and the same multivalued part, $\G^{-1}(\M\cap\N);$ since $\ker C_B^* \subseteq\G^{-1}(\M)$ and $\ran C_B \subseteq\G^{-1}(\N)$ then $P_{\ker C_B^*} \subseteq  P_{\G^{-1}(\M), \G^{-1}(\N)}.$

Finally, write $\St:=\G^{-1}(\M\cap\N)$ and $P_0=P_{\ker C_B^*\ominus (\ker C_A^*\cap\ker C_B^*) }.$ We claim that $$P_{\ker C_B^*}\hat{\dotplus}( \{0\}\times \St)=P_0 \ \hat{\oplus} \ (\{0\}\times \St).$$ Again, both relations have the same domain and multivalued part. Also, $I_{\ran P_0}\subseteq I_{\ker C_B^*}\subseteq P_{\ker C_B^*}$ and $\ker P_0\times \{0\}=(\ran C_B\times \{0\})\hat{\dotplus}((\ker C_A^*\cap\ker C_B^*)\times \{0\})\subseteq P_{\ker C_B^*}\hat{\dotplus}( \{0\}\times \St).$ Then the inclusion $\supseteq$ holds, and the identity follows.

Conversely, assume that there exist $\G\in L(\HH)$ positive (semi-definite) with $\ran \G=\dom E$ such that $\G^{-1}E\G \in \MP(\HH)$ with $P_0:=(\G^{-1}E\G)_0 \in \mc P.$  By hypothesis, $\St:=\mul \G^{-1}E\G \subseteq \ker P_0.$ Then $\G^{-1}E\G=I_{\ran P_0} \ \hat{+} \ (\ker P_0\times \St)=P_{\ran P_0\oplus\St, \ker P_0}.$ From $\dom \G^{-1}E\G=\HH,$ it follows that $\ran E\subseteq \ran \G=\dom E.$ Also, $\ker \G \subseteq \St.$ In fact, $\ker \G=\mul \G^{-1}\subseteq \mul \G^{-1}E\G=\mul (P_0 \ \hat{+} \ (\{0\} \times \St))=\St.$ Then
\begin{eqnarray}\label{EE}
E&=&I_{\ran \G} EI_{\ran \G}=\G \G^{-1} E \G\G^{-1}=\G P_{\ran P_0\oplus\St, \ker P_0} \G^{-1}\nonumber\\
&=&P_{\G(\ran P_0\oplus\St), \G(\ker P_0)},
\end{eqnarray}
where we apply Lemma \ref{D} to $\tilde{\G}:=\G^{-1},$ as
$\ran \G^{-1}=\HH=\dom P_{\ran P_0\oplus\St, \ker P_0}$ and $\mul \G^{-1}=\ker \G\subseteq \ker P_0.$ Therefore, $E\in \Sp(\HH).$ Furthermore, $E$ is an operator range because $\ran E$ and $\ker E$ are operator ranges.
 \end{proof}

\begin{cor} Let  $\M$ and $\N$ be operator ranges such that $\M \cap \N$ is closed in $\M+\N.$  Then
$$\PMN X= X(P_0 \ \hat{\oplus} \ (\{0\} \times \St)),$$ 
where $X,P_0\in L(\overline{\M+\N})$ are positive, $X$ is a quasi-affinity, $P_0$ is a projection and $\St$ is a closed subspace.
\end{cor}

\subsection*{Acknowledgments}
M.~Contino, A.~Maestripieri  and S.~Marcantognini were
supported by CONICET PIP 11220200102127CO.  
M.~Contino was supported by María Zambrano Postdoctoral Grant CT33/21 at Universidad Complutense de Madrid financed by the Ministry of Universities with Next Generation EU funds and
by Grant CEX2019-000904-S funded by MCIN/AEI/10.13039/501100011033. 



\end{document}